\documentclass[11pt]{amsart}

\usepackage{latexsym}
\usepackage{amssymb}
\usepackage{amsmath}
\usepackage{color}
\usepackage{comment}

\usepackage[all]{xy}

\newtheorem{theorem}{Theorem}[section]
\newtheorem{lemma}[theorem]{Lemma}

\newtheorem{corollary}[theorem]{Corollary}

\newtheorem{remark}[theorem]{Remark}

\newcommand\supp{\mathop{\rm supp}}

\newcommand\esssup{\mathop{\rm esssup}}

\newcommand\id{\mathop{\rm id}}

\newcommand\nph{\varphi}

\newcommand\hh{\mathop{\rm h}}
\newcommand\eh{\mathop{\rm eh}}

\newcommand\vn{\mathop{\rm VN}}

\newcommand\cb{\mathop{\rm cb}}
\newcommand\ah{\mathop{A_{\hh}}}

\newcommand\ahg{\mathop{\ah(G)}}

\newcommand{\cl}[1]{\mathcal{#1}}
\newcommand{\bb}[1]{\mathbb{#1}}

\newcommand{\bG}{\mathbb{G}}
\newcommand{\G}{\mathbb{G}}

\newcommand{\norm}[1]{\| #1\|}

\newcommand{\cM}{\mathcal{M}}

\newcommand{\bM}{\mathbb{M}}

\newcommand{\m}{\textsf{m}}

\newcommand{\ignore}[1]{{ }}
\begin{document}

\title{Completely bounded maps and invariant subspaces}

\author{M. Alaghmandan}
\address{School of Mathematics and Statistics, Carleton University, Ottawa, ON, Canada H1S 5B6}
\email{mahmood.alaghmandan@carleton.ca}

\author{I. G. Todorov}
\address{Mathematical Sciences Research Centre, Queen's University Belfast, Belfast BT7 1NN, United Kingdom, and 
School of Mathematical Sciences, Nankai University, 300071 Tianjin, China}
\email {i.todorov@qub.ac.uk}

\author{L. Turowska}
\address{Department of Mathematical Sciences,
Chalmers University of Technology and  the University of Gothenburg,
Gothenburg SE-412 96, Sweden}
\email{turowska@chalmers.se}

\date{27 January 2019}

\maketitle

\begin{abstract}
We provide a description of certain invariance properties of completely bounded bimodule maps
in terms of their symbols. If $\bG$ is a locally compact quantum group,
we characterise the completely bounded $L^{\infty}(\bG)'$-bimodule maps that send
$C_0(\hat\bG)$ into $L^{\infty}(\hat\bG)$ in terms of the properties of
the corresponding elements of the normal Haagerup tensor product $L^{\infty}(\bG) \otimes_{\sigma\hh} L^{\infty}(\bG)$.
As a consequence,
we obtain an intrinsic characterisation of the normal completely bounded $L^{\infty}(\bG)'$-bimodule maps that leave
$L^{\infty}(\hat\bG)$ invariant, extending and unifying results, formulated in the current literature separately for the
commutative and the co-commutative cases.
\end{abstract}

\section{Introduction}

Let $\cl M$ be a von Neumann algebra, acting on a Hilbert space $H$, with commutant $\cl M'$.
The \emph{normal Haagerup tensor product} $\cl M\otimes_{\sigma\hh}\cl M$ was introduced by
E. G. Effros and A. Kishimoto in \cite{ek}, where it was shown
that to each of its element $\chi$ there corresponds, in a canonical way, 
a completely bounded $\cl M'$-bimodule map $\Phi_{\chi}$ on the space $\cl B(H)$ of all bounded linear operators on $H$.
Let us call $\chi$ the \emph{symbol} of the map $\Phi_{\chi}$.
The transformation $\Phi_{\chi}$ is in addition weak* continuous precisely when $\chi$ belongs to the, smaller,
\emph{extended Haagerup tensor product} $\cl M\otimes_{\eh}\cl M$, introduced by U. Haagerup \cite{haagerup}.
The present paper is concerned with the question of how specific invariant subspace properties of the map
$\Phi_{\chi}$ are reflected in its symbol $\chi$.

As a first motivation, we single out Herz-Schur multipliers of the Fourier algebra $A(G)$ of a locally compact group $G$,
introduced by J. de Canniere and U. Haagerup \cite{ch}. These objects have played a prominent role
in operator algebra theory, in particular in the study of approximation properties of
the von Neumann algebra $\vn(G)$ and the reduced group C*-algebra $C_r^*(G)$ of $G$ \cite{bo, hk, knudby}. 
The numerous applications they have found were 
facilitated to a great extent by the description, due to J. E. Gilbert and to M. Bo$\dot{\textrm{z}}$ejko and G. Fendler
\cite{bf_arch}, which identifies them with a part of the Schur multipliers on the direct product $G\times G$.
Recall that the space of Schur multipliers \cite{haagerup, pa, pi, t_serdica} coincides with the 
extended Haagerup tensor product $L^{\infty}(G) \otimes_{\eh} L^{\infty}(G)$,
where $L^{\infty}(G)$ is identified with the algebra of operators of multiplication by essentially bounded functions on $G$,
acting on the Hilbert space $L^2(G)$.
Herz-Schur multipliers can be identified with
those completely bounded weak* continuous $L^{\infty}(G)$-bimodule maps that leave $\vn(G)$ invariant.
As $L^{\infty}(G) \otimes_{\eh} L^{\infty}(G)$ can be naturally embedded in $L^{\infty}(G\times G)$,
a natural problem, addressed in \cite{spronk} and \cite{nrs}, is that
of characterising the functions $\nph : G\times G \to \bb{C}$
that are symbols of Herz-Schur multipliers.

Exchanging the roles of $L^{\infty}(G)$ and $\vn(G)$, one
may consider the space of completely bounded weak* continuous $\vn(G)'$-bimodule maps on $\cl B(L^2(G))$.
Since $\vn(G) \otimes_{\eh} \vn(G)$ naturally embeds in $\vn(G\times G)$,
to each such map there corresponds canonically an element of the Neumann algebra $\vn(G\times G)$.
Those such maps that leave $L^{\infty}(G)$ invariant
were characterised by M. Neufang \cite{neufang} (see also \cite{nrs})
as arising in a canonical fashion from bounded complex measures on $G$.
In \cite{bi}, under the restriction that $G$ be weakly amenable,
we provided an intrinsic characterisation of the
symbols from $\vn(G) \otimes_{\eh} \vn(G)$ whose maps leave $L^{\infty}(G)$ invariant,
showing that, when viewed as elements of $\vn(G\times G)$,  
they are precisely those supported, in the sense of P. Eymard \cite{eymard}, on the anti-diagonal of $G$.

The invariance properties in the two cases described above
were placed in the same context in \cite{junge-Neufang},
where the authors
showed that, if $\G$ is a locally compact quantum group with underlying von Neumann algebra $L^{\infty}(\G)$
and a dual quantum group $\hat\G$ with underlying von Neumann algebra $L^{\infty}(\hat\G)$,
the completely bounded weak* continuous
$L^{\infty}(\G)'$-bimodule maps that leave $L^{\infty}(\hat\G)$ invariant correspond in a canonical fashion
to the completely bounded left multipliers of the predual $L^1(\hat\G)$ of $L^{\infty}(\hat\G)$
(we direct the reader to \cite{junge-Neufang} for the definitions and further details).
However, the problem of characterising intrinsically the elements of
$L^{\infty}(\G) \otimes_{\eh} L^{\infty}(\G)$
which are symbols of maps leaving the von Neumann algebra $L^{\infty}(\hat\G)$ invariant, has not been addressed in the literature.

One of the main aims of the present paper is to exhibit such a characterisation.
We unify results in the literature that are currently stated in the two extreme cases -- for commutative and co-commutative locally compact
quantum groups -- arriving at a condition that captures both simultaneously.
In fact, our results go beyond this aim, as we are able to drop the requirement of
weak* continuity, and characterise intrinsically those elements of the normal Haagerup tensor product
$L^{\infty}(\G) \otimes_{\sigma \hh} L^{\infty}(\G)$ that give rise to completely bounded
(but not necessarily weak* continuous) maps  on $\cl B(L^2(\G))$
sending the reduced C*-algebra $C_0(\hat\G)$ of $\hat\G$ into $L^{\infty}(\hat\G)$.

Specialising to the case of co-commutative quantum groups, we obtain a generalisation of our previous results from \cite{bi},
removing the assumption of weak amenability imposed therein.
On the other hand, specialising to the case of commutative quantum groups,
we obtain the first, to the best of our knowledge, rigorous characterisation
of the Schur multipliers $\nph : G\times G\to \bb{C}$ whose corresponding map on $\cl B(L^2(G))$ leaves
$\vn(G)$ invariant, in terms of the function $\nph$ alone; the result states that this happens if and only if, for every $r\in G$,
the equality $\nph(sr,tr) = \nph(s,t)$ holds for marginally almost all $(s,t)\in G\times G$.

Our main result in the setting of quantum groups, 
Theorem \ref{t:(m)-iff-VN-invariant}, is obtained as a consequence of a more
general statement (Theorem \ref{th_chargenn}) that expresses invariance properties of completely bounded
bimodule maps in terms of their symbol.
This general viewpoint is presented in Section \ref{s_ci}, after obtaining some preparatory
results. It is applied, in Section \ref{s:LCQG}, to the case of locally compact quantum groups,
while Sections \ref{ss_coc} and \ref{ss_com} contain the further applications to
co-commutative and commutative quantum groups, respectively.

\medskip

We finish this section by setting basic notation.
We denote by $\cl B(H)$ the algebra of all bounded linear operators acing on a Hilbert space $H$.
If $\cl M\subseteq \cl B(H)$ is a von Neumann algebra, we let $\cM_*$ denote its predual,
consisting of all normal functionals on $\cl M$, and equip it
with the operator space structure arising from its natural embedding into $\cl M^*$.
We denote by ${\rm CB}_{\cl M}(\cl B(H))$ (resp. ${\rm CB}_{\cl M}^{\sigma}(\cl B(H))$)
the operator space of all completely bounded (resp. weak* continuous, or normal, completely bounded) $\cl M$-bimodule maps on
$\cl B(H)$.
The algebraic tensor product of vector spaces $\cl X$ and $\cl Y$ is denoted by $\cl X\odot \cl Y$,
the space of all $n$ by $n$ matrices -- by $\bM_n$, and the tensor product $\cl X\odot \bM_n$ is written as $\bM_n(\cl X)$
and identified with the space of all $n$ by $n$ matrices with entries in $\cl X$.
If $\cl X$ and $\cl Y$ are operator spaces, we denote by $\cl X\otimes_{\hh} \cl Y$ the Haagerup tensor product of $\cl X$ and $\cl Y$,
and by $\cl X\hat\otimes \cl Y$ their operator projective tensor product.
If $\cl X$ and $\cl Y$ are moreover dual operator spaces, $\cl X\bar\otimes \cl Y$ stands for their weak* spatial tensor product.
If $\nph : \cl X\times\cl Y \to \cl Z$ (where $\cl Z$ is another operator space) and $n,m\in \bb{N}$, we let
$\nph^{(m,n)} : \bM_n(\cl X)\times \bM_m(\cl Y) \to \bM_{nm}(\cl Z)$ be the ampliation of $\nph$, given by
$\nph^{(m,n)}((x_{i,j}),(y_{p,q})) = (\nph(x_{i,j},y_{p,q})_{i,p,j,q}$.
The identity operator is denoted by $1$, and
we let $\cl M\otimes 1 = \{a\otimes 1 : a\in \cl M\}$; it will be clear from the context on which Hilbert space $1$ acts.
We will use throughout the paper basic results from operator space theory, and we refer the reader to the monographs
\cite{blm, er, pa, pi} for the necessary background.

%%%%%%%%%%%%%%%%         	CHARACTERISATION OF INVARIANCE         %%%%%%%%%%%%%%%%%%

\section{A characterisation of invariance}\label{s_ci}

Let $\cl X$ and $\cl Y$ be dual operator spaces and
${\rm Bil}_{\sigma}(\cl X,\cl Y)$
be the operator space of all normal (i.e. separately weak* continuous)
completely bounded bilinear forms on $\cl X\times \cl Y$.
The \emph{normal Haagerup tensor product} \cite{blm, ek} of $\cl X$ and $\cl Y$
is the operator space dual of ${\rm Bil}_{\sigma}(\cl X,\cl Y)$:
$$\cl X\otimes_{\sigma\hh}\cl Y \stackrel{def}{=} {\rm Bil}_{\sigma}(\cl X,\cl Y)^*.$$
It is characterised by the following universal property: for every dual operator space $\cl Z$ and every
normal completely bounded bilinear map $\phi : \cl X\times \cl Y \to \cl Z$,
there exists a (unique) normal completely bounded linearisation $\tilde{\phi} : \cl X\otimes_{\sigma\hh}\cl Y \to \cl Z$.
It is easy to see that the normal Haagerup tensor product is functorial: if $\cl X_1$ and $\cl Y_1$
are dual operator spaces and $\phi : \cl X \to \cl X_1$ and $\psi : \cl Y\to \cl Y_1$ are normal completely bounded maps
then there exists a (unique) normal completely bounded map
$\phi\otimes\psi : \cl X\otimes_{\sigma\hh}\cl Y \to \cl X_1\otimes_{\sigma\hh}\cl Y_1$ such that
\begin{equation}\label{eq_elte}
(\phi\otimes\psi)(a\otimes b) = \phi(a)\otimes\psi(b), \ \ \ a\in \cl X, b\in \cl Y.
\end{equation}
\begin{comment}
Indeed, the map 
$$\theta : {\rm Bil}_{\sigma}(\cl X_1,\cl Y_1) \to {\rm Bil}_{\sigma}(\cl X,\cl Y)$$
given by
$$\theta(F)(a,b) = F(\phi(a),\psi(b)), \ \ \ a\in \cl X, b\in \cl Y,$$
is well-defined and completely bounded.
Setting $\phi \otimes\psi \stackrel{def}{=}\theta^*$, we
have that if $a\in \cl X$, $b\in \cl Y$ and $F\in {\rm Bil}_{\sigma}(\cl X_1,\cl Y_1)$
then
$$\langle (\phi \otimes\psi)(a\otimes b),F\rangle = \langle a\otimes b,\theta(F)\rangle
= F(\phi(a),\psi(b)) = \langle \phi(a) \otimes\psi(b),F\rangle;$$
thus, (\ref{eq_elte}) holds true. In addition, $\phi \otimes\psi$ is normal since it is a dual map.
\end{comment}

Let $\cl M$ be a von Neumann algebra acting on a Hilbert space $H$.
There exists a canonical complete isometry, that is also a weak* homeomorphism,
from $\cl M\otimes_{\sigma \hh} \cl M$ onto the space ${\rm CB}_{\cl M'}(\cl B(H))$
of all completely bounded $\cl M'$-bimodule maps on $\cl B(H)$ \cite{ek}.
For $\chi\in \cl M\otimes_{\sigma\hh}\cl M$, we let
$\Phi_{\chi} : \cl B(H)\to \cl B(H)$ be its corresponding map; we have
that
\begin{equation}\label{eq:Phi}
\Phi_{a\otimes b}(x) = axb, \ \ \ x\in \cl B(H), \ a, b \in \cl M.
\end{equation}
We call $\chi$ the \emph{symbol} of $\Phi_{\chi}$.
If $\xi,\eta\in H$, $x\in \cl B(H)$
and $F_{x,\xi,\eta} : \cl M\times \cl M \to \bb{C}$ is the (normal completely bounded) bilinear form
given by $F_{x,\xi,\eta}(a,b) = \langle axb\xi,\eta\rangle$, then \cite[(2.5)]{ek}
$$\langle \chi, F_{x,\xi,\eta}\rangle = \langle \Phi_{\chi}(x)\xi,\eta\rangle, \ \ \ \chi\in \cl M\otimes_{\sigma\hh}\cl M.$$
It follows that the map $\chi \to \Phi_{\chi}$ is continuous when $\cl M\otimes_{\sigma\hh}\cl M$ is equipped with its
weak* topology and ${\rm CB}_{\cl M'}(\cl B(H))$ -- with the point-weak operator topology.

Recall that the \emph{extended Haagerup tensor product} $\cl M\otimes_{\eh}\cl M$ \cite{er}
coincides with the \emph{weak* Haagerup tensor product} $\cl M\otimes_{w^*\hh}\cl M$ \cite{blecher},
and can be canonically identified with the operator space ${\rm CB}_{\cl M'}^{\sigma}(\cl B (H))$ of all
normal completely bounded $\cl M'$-bimodule maps on $\cl B(H)$: for each element
$\chi\in \cl M\otimes_{\eh}\cl M$, there exist families
$(a_i)_{i\in \bb{I}}, (b_i)_{i\in \bb{I}} \subseteq \cl M$
such that the series $\sum_{i\in \bb{I}} a_i a_i^*$ and $\sum_{i\in \bb{I}} b_i^* b_i$ are weak* convergent and the
corresponding map $\Phi_{\chi} : \cl B(H)\to \cl B(H)$ is given by
$$\Phi_{\chi}(x) = \sum_{i\in \bb{I}} a_i x b_i, \ \ \ x\in \cl B(H),$$
where the series converges in the weak* topology of $\cl B(H)$.
We write $\chi \sim \sum_{i\in \bb{I}} a_i\otimes b_i$.

Note that the canonical inclusion $\cl M\otimes_{\eh}\cl M \subseteq \cl M\otimes_{\sigma\hh}\cl M$ is completely isometric.
If $\cl N\subseteq \cl B(H)$ is a(nother) von Neumann algebra, we let
$${\rm CB}_{\cl M'}^{\sigma, \cl N}(\cl B(H)) = \{\Phi \in {\rm CB}_{\cl M'}^{\sigma}(\cl B(H)) : \Phi(\cl N)\subseteq\cl N\}.$$

Let $H$ be a Hilbert space, $\cl M\subseteq \cl B(H)$ be a von Neumann algebra and $f,g\in \cl M_*$.
By the functoriality of the weak* spatial tensor product,
$L_f \stackrel{def}{=} f\otimes\id$ and $R_g \stackrel{def}{=} \id\otimes g$ are well-defined normal completely bounded maps
from $\cl M\bar\otimes\cl M$ into $\cl M$; note that
$$L_f(a\otimes b) = f(a)b \mbox{ and } R_g(a\otimes b) = g(b) a, \ \ \ a,b\in \cl M,$$
$\|L_f\|_{\cb} = \|f\|$ and $\|R_g\|_{\cb} = \|g\|$ (see \cite{tom}).
By the functoriality of the normal Haagerup tensor product,
$$L_f\otimes R_g : (\cl M\bar\otimes\cl M)\otimes_{\rm\sigma h}(\cl M\bar\otimes\cl M)
\to \cl M\otimes_{\sigma\hh} \cl M$$
is a weak* continuous completely bounded map (see \cite{er}).

Let $\m : \cl M\times\cl M\to\cl M$ denote operator multiplication.
The map $\m$ extends uniquely to a weak* continuous completely contractive map
(denoted in the same way) $\m: \cl M\otimes_{\rm \sigma h}\cl M \to \cl M$ \cite{er}.
We let $\m_* : \cl M_* \to \cl M_* \otimes_{\eh} \cl M_*$ be its predual.

\begin{lemma}\label{l:T}
For every $\psi\in (\cl M\bar\otimes\cl M)\otimes_{\rm \sigma h}(\cl M\bar\otimes\cl M)$ there exists (a unique)
 $\mathcal T(\psi) \in\cl M\bar\otimes\cl M\bar\otimes\cl M$ such that
 \[
 \langle \mathcal T(\psi), f\otimes \omega \otimes g\rangle=\langle {\rm m}((L_f\otimes R_g)(\psi)),\omega\rangle,\quad f,g,\omega \in \cl M_*.
 \]
 Moreover, the map
$$\mathcal T : (\cl M\bar\otimes\cl M)\otimes_{\rm \sigma h}(\cl M\bar\otimes\cl M)\to \cl M\bar\otimes\cl M\bar\otimes\cl M$$
is linear, contractive and weak*-continuous, and
 \begin{equation}\label{T:XY}
 \mathcal T(\chi_1\otimes \chi_2)=(\chi_1\otimes 1)(1\otimes \chi_2)
 \end{equation}
  for all $\chi_1, \chi_2 \in\cl M\bar\otimes\cl M$.
\end{lemma}

\begin{proof}
Let $\psi\in (\cl M\bar\otimes\cl M)\otimes_{\rm \sigma h}(\cl M\bar\otimes\cl M)$ and
$F_\psi: \cl M_*\times\cl M_*\times\cl M_*\to\mathbb C$ be the trilinear map given by
\begin{equation}\label{eq_Fpsi}
F_{\psi}(f,\omega,g) = \langle \m((L_f\otimes R_g)(\psi)),\omega\rangle, \ \ \ f,g,\omega \in\cl M_*.
\end{equation}
Let $  X = (f_{i,j})_{i,j}\in \bM_n(\cl M_*)$, $Y = (g_{k,l})_{k,l} \in \bM_m(\cl M_*)$, and
$$L_X : \cl M\bar{\otimes} \cl M \to M_n(\cl M), \ 
R_Y : \cl M\bar{\otimes} \cl M \to M_m(\cl M)$$ 
be the maps given by
$$L_{X}(\chi) = (L_{f_{i,j}}(\chi))_{i,j}, \ R_{Y}(\chi) = (R_{g_{k,l}}(\chi))_{k,l}, \ \ \ \chi \in \cl M\bar\otimes \cl M.$$
Standard arguments show that $L_X$ and $R_Y$ are completely bounded and 
$$\|L_X\|_{\rm cb} \leq \norm{X}_{\bM_n(\cl M_*)}, \ \ \|R_Y\|_{\rm cb}\leq \norm{Y}_{\bM_m(\cl M_*)}.$$

By   \cite[p. 149]{er}, the map
$$L_X\otimes R_Y :
(\cl M\bar\otimes\cl M)\otimes_{\rm \sigma h}  (\cl M\bar\otimes\cl M) \to \bM_n(\cl M)\otimes_{\rm \sigma h} \bM_m(\cl M)$$
is completely bounded with $\|L_X\otimes R_Y\|_{\cb} \leq \norm{X}_{\bM_n(\cl M_*)}\norm{Y}_{\bM_m(\cl M_*)}$.

By \cite[Theorem 6.1]{er},
the canonical shuffle map
\begin{eqnarray*}
S_\sigma : \bM_n(\cl M)\otimes_{\rm \sigma h}\bM_m(\cl M)
& = &
(\cl M\bar\otimes \bM_n)\otimes_{\rm \sigma h}(\cl M\bar\otimes \bM_m)\\
& \to & (\cl M\otimes_{\rm \sigma h}\cl M)\bar\otimes (\bM_n\otimes_{\rm \sigma h} \bM_m)
\end{eqnarray*}
is a complete contraction.
Since the minimal norm is dominated by the Haagerup norm, we have, on the other hand, that the map
\[
\theta: \bM_n\otimes_{\rm \sigma h} \bM_m= \bM_n\otimes_{\rm  h}\bM_m\to \bM_{nm}
\]
is a  complete contraction.
It follows that the canonical map
$$\tau : \bM_n(\cl M)\otimes_{\rm \sigma h}\bM_m(\cl M) \to {\bM_{nm}(\cl M\otimes_{\rm\sigma h}\cl M)},$$
induced by $S_{\sigma}$ and $\theta$,
is a complete contraction.
Therefore,  for each $p\in \bb{N}$ and $Z = (\omega_{s,t})_{s,t} \in \bM_p(\cl M_*)$, we have
\begin{eqnarray*}
& & \|F_\psi^{(m,n,p)}(X,Y,Z)\|_{\bM_{mnp}}\\
& = &\|\langle ((L_{f_{i,j}}\otimes R_{g_{k,l}})(\psi), \m_*^{(p)}(\omega_{s,t})\rangle\|_{M_{mnp}}\\
& \leq & \|((L_{f_{i,j}}\otimes R_{g_{k,l}})(\psi))_{i,j,k,l}\|_{\bM_{nm}(\cl M\otimes_{\rm\sigma h}\cl M)}
\|(\m_*(\omega_{s,t}))_{s,t}\|_{\bM_p(\cl M_*\otimes_{\rm eh}\cl M_*)}\\
& \leq & \| (L_X\otimes R_Y)(\psi)\|_{\bM_{n}(\cl M)\otimes_{\rm\sigma h}\bM_m(\cl M)}\|Z\|\leq \|\psi\|\|X\|\|Y\|\|Z\|.
\end{eqnarray*}
Thus,
the map $F_\psi$  is a completely bounded trilinear form on $\cl M_*\times\cl M_*\times\cl M_*$, with
\begin{equation}\label{cb}
\|F_\psi\|_{\rm cb}\leq  \|\psi\|.
\end{equation}
It therefore linearises to a bounded linear functional on $\cl M_*\hat\otimes\cl M_*\hat\otimes\cl M_*$,
which we denote again by $F_{\psi}$.
Since $(\cl M_*\hat\otimes\cl M_*\hat\otimes\cl M_*)^* \equiv \cl M\bar\otimes\cl M \bar\otimes\cl M$,
there exists $\cl T(\psi) \in \cl M\bar\otimes\cl M \bar\otimes\cl M$ such that
\begin{equation}\label{eq_dualw}
F_{\psi}(w) = \langle \cl T(\psi),w\rangle, \ \ \ w\in \cl M_*\hat{\otimes}\cl M_*\hat\otimes\cl M_*.
\end{equation}
By (\ref{cb}),
\begin{eqnarray*}
\|\cl T(\psi)\|
& = & \sup\{|\langle \cl T(\psi),w\rangle| : \|w\|\leq 1\}
= \sup\{\|F_{\psi}(w)\| : \|w\|\leq 1\}\\
& = &
\|F_{\psi}\|\leq \|\psi\|.
\end{eqnarray*}
The uniqueness of $\cl T(\psi)$ follows from the density of
$\cl M_*\odot \cl M_*\odot\cl M_*$ in $\cl M_*\hat{\otimes}\cl M_*\hat\otimes\cl M_*$.

We show that $\cl T$ is weak* continuous. To this end,
for each element $w$ of $\cl M_*\hat{\otimes}\cl M_*\hat\otimes\cl M_*$, define a linear functional  $E_w$
on $(\cl M\bar\otimes\cl M)\otimes_{\rm \sigma h}(\cl M\bar\otimes\cl M)$
by
\[
E_w( \psi) = \langle \cl T(\psi), w\rangle, \ \ \ \psi \in (\cl M\bar\otimes\cl M)\otimes_{\rm \sigma h}(\cl M\bar\otimes\cl M).
\]
It clearly suffices to prove that $E_w$ is weak* continuous, for each $w$.
By \cite[p 75]{s}, it suffices to show that
the kernel of $E_w$ is weak* closed which, by virtue of the
Krein-Shmulian Theorem \cite[p 152]{s}, is equivalent to the fact that
the intersection of $\ker(E_w)$ with every norm closed ball of $(\cl M\bar\otimes\cl M)\otimes_{\rm \sigma h}(\cl M\bar\otimes\cl M)$ is  weak* closed.

Fix $w \in \cl M_*\hat{\otimes}\cl M_*\hat\otimes\cl M_*$. Let $C > 0$,
$\psi \in (\cl M\bar\otimes\cl M)\otimes_{\rm \sigma h}(\cl M\bar\otimes\cl M)$ and  $(\psi_{\alpha})_{\alpha\in \bb{A}}$ be a net
in $\ker(E_w)$ that converges to   $\psi$ in the weak* topology, such that
$\|\psi_{\alpha}\|\leq C$,  $\alpha \in \bb{A}$. Fix $\epsilon > 0$.
Let $w_0 \in \cl M_* \odot \cl M_* \odot \cl M_*$ be such that
$\|w - w_0\| < \epsilon/3C$.
By the weak* continuity of $L_f\otimes R_g$ and $\m$, (\ref{eq_Fpsi}) and (\ref{eq_dualw})
imply that there exists $\alpha_0\in \bb{A}$ such that
\[
|E_{w_0}(\psi_\alpha) - E_{w_0}(\psi)|= |F_{\psi_{\alpha}}(w_0) - F_{\psi}(w_0)| < \epsilon/3, \ \ \ \alpha\geq \alpha_0.
\]
An $\epsilon/3$-argument now implies that
\[
|E_{w}(\psi)| =|E_{w}(\psi_\alpha) - E_{w}(\psi)|= |F_{\psi_{\alpha}}(w) - F_{\psi}(w)| < \epsilon, \ \ \ \alpha\geq \alpha_0.
\]
It follows that $E_w(\psi)=0$ and the weak* continuity of $\cl T$ is established.

Finally, we show (\ref{T:XY}). Assume first that $\chi_1 = a\otimes b$ and $\chi_2 = c\otimes d$ where $a,b,c,d\in \cl M$. Then
\begin{eqnarray*}
\langle \m\left((L_f\otimes R_g)(\chi_1 \otimes \chi_2) \right),\omega\rangle
& = & \langle\m(L_f(\chi_1)\otimes R_g(\chi_2)),\omega\rangle\\
& = & \langle L_f(\chi_1)R_g(\chi_2),\omega\rangle\\
& = & f(a) g(d) \; \langle bc ,\omega\rangle\\
& = & \langle a\otimes bc\otimes d, f\otimes \omega \otimes g\rangle\\
& = & \langle (\chi_1\otimes 1)(1\otimes \chi_2), f\otimes \omega \otimes g\rangle,
\end{eqnarray*}
giving $\cl T(\chi_1\otimes \chi_2)=(\chi_1\otimes 1)(1\otimes \chi_2)$.
By linearity,
the equality holds for any $\chi_1$, $\chi_2$ in the algebraic tensor product $\cl M\odot\cl M$.
As $\cl T$ is weak* continuous and the multiplication in $\cl M\bar\otimes\cl M\bar\otimes\cl M$ is separately weak*
continuous, we obtain the equality (\ref{T:XY}) for any $\chi_1, \chi_2 \in\cl M\bar\otimes\cl M$.
\end{proof}

\noindent {\bf Remark. } Intuitively, the map $\cl T$ from Lemma \ref{l:T}
\lq\lq multiplies the two middle variables'' in the four-term tensor product,
leaving the outer variables intact,
thus producing a three-variable element.

\medskip

Let $H$ be a Hilbert space, and $\cl M$ and $\cl N$ be von Neumann algebras acting on $H$.
We fix for the rest of the section a unitary operator $W\in \cl B(H)\bar\otimes\cl N$
(resp. $V\in \cl N \bar\otimes \cl B(H)$), and define maps
$$\Gamma_W : \cl M\to \cl B(H\otimes H) \ \mbox{ and } \ \Gamma'_V : \cl M \to \cl B(H\otimes H)$$
by letting
\begin{equation}\label{eq_gammas}
\Gamma_W(a) = W^*(1\otimes a)W  \ \mbox{ and } \  \Gamma'_V(b) = V(b\otimes 1)V^*.
\end{equation}
Clearly, $\Gamma_W$ and $\Gamma'_V$ are normal completely bounded maps; thus,
there exists a normal completely bounded map
$$\Gamma_W\otimes\Gamma'_V : \cl M\otimes_{\sigma\hh} \cl M \to (\cl B(H)\bar\otimes \cl B(H))\otimes_{\sigma\hh} (\cl B(H)\bar\otimes \cl B(H))$$
such that
$$(\Gamma_W\otimes\Gamma'_V) (a\otimes b) = \Gamma_W(a)\otimes \Gamma'_V(b), \ \ \ a,b\in \cl M.$$

For an operator $T\in \cl B(H\otimes H)$, let
$T_{1,2} = T\otimes 1$ and $T_{2,3} = 1 \otimes T$. 
%It is clear that $W_{1,2}$ and $V_{2,3}$ are unitary operators on $H\otimes H\otimes H$.
Recall that, for $\chi \in \cl M\otimes_{\sigma\hh} \cl M$,
we denote by $\Phi_\chi$ the corresponding completely bounded mapping on ${\cl B}(H)$.

\begin{lemma}\label{l_ident}
Let $H$ be a Hilbert space, $\cl M$ and $\cl N$ be von Neumann algebras acting on $H$,
$\chi\in \cl M\otimes_{\sigma\hh}\cl M$,
$W \in \cl B(H)\bar\otimes\cl N$, $V \in \cl N \bar\otimes \cl B(H)$ be unitary operators, and $f,g\in \cl B(H)_*$. 
Then
\begin{equation}\label{eq_slices}
(f\otimes \id\otimes g)(W_{1,2} (\cl T\circ (\Gamma_W\otimes\Gamma'_V)(\chi) )V_{2,3})
= \Phi_{\chi}((f\otimes\id)(W)(\id\otimes g)(V)).
\end{equation}
\end{lemma}
\begin{proof}
Assume first that $\chi = a\otimes b$, where $a,b\in \cl M$.
Using Lemma \ref{l:T}, we have
\begin{eqnarray*}
\mathcal T\circ(\Gamma_W\otimes \Gamma'_V)(\chi)
& = &
\mathcal T(\Gamma_W(a) \otimes \Gamma'_V(b))\\
& = &
\mathcal T((W^*(1 \otimes a)W) \otimes (V(b \otimes 1)V^*))\\
& = &
W_{1,2}^*(1 \otimes a \otimes 1) W_{1,2} V_{2,3}  (1 \otimes b \otimes 1) V^*_{2,3}.
\end{eqnarray*}
Hence
$$W_{1,2} (\mathcal T\circ(\Gamma_W\otimes \Gamma_V')(\chi)) V_{2,3} =
(1 \otimes a \otimes 1) W_{1,2} V_{2,3}  (1 \otimes b \otimes 1).$$
Thus, in order to establish (\ref{eq_slices}), it suffices to show that, whenever $T\in \cl B(H)\bar\otimes\cl N$ and $S\in \cl N\bar\otimes\cl B(H)$,
we have
\begin{equation}\label{eq_fwgv}
(f\otimes \id\otimes g)((1 \otimes a \otimes 1) T_{1,2} S_{2,3}  (1 \otimes b \otimes 1))
= \Phi_{\chi}((f\otimes\id)(T)(\id\otimes g)(S)).
\end{equation}
To this end, assume first that $T = T_1\otimes T_2$ and $S = S_1\otimes S_2$, where $T_1,S_2\in \cl B(H)$ and
$T_2,S_1\in \cl N$.
Then
\begin{eqnarray*}
& &
(f\otimes \id\otimes g)((1 \otimes a \otimes 1) T_{1,2} S_{2,3}  (1 \otimes b \otimes 1))\\
& = &
(f\otimes \id\otimes g)((1 \otimes a \otimes 1)(T_1\otimes T_2S_1\otimes S_2)(1 \otimes b \otimes 1))\\
& = &
(f\otimes \id\otimes g)(T_1\otimes \Phi_{\chi}(T_2S_1)\otimes S_2)\\
& = &
f(T_1) g(S_2)\Phi_{\chi}(T_2S_1)
=
\Phi_{\chi}((f\otimes\id)(T)(\id\otimes g)(S)).
\end{eqnarray*}
By linearity, (\ref{eq_fwgv}) holds true if $T$ (resp. $S$) belongs to the algebraic tensor product $\cl B(H)\odot \cl N$
(resp. $\cl N \odot \cl B(H)$).
Equation (\ref{eq_fwgv}) now follows from the weak* continuity of $\Phi_{\chi}$.

By linearity, (\ref{eq_slices}) holds whenever $\chi$ is an element of the algebraic tensor product
$\cl M\odot\cl M$. Now assume that $\chi\in \cl M\otimes_{\sigma\hh}\cl M$ is arbitrary.
Let $(\chi_{\alpha})_{\alpha\in \bb{A}}\subseteq \cl M\odot\cl M$ be a net converging to $\chi$ in the weak* topology of
$\cl M\otimes_{\sigma\hh}\cl M$. Then 
$$\Phi_{\chi_{\alpha}}((f\otimes\id)(W)(\id\otimes g)(V))\to_{\alpha\in \bb{A}}^{\rm WOT} \Phi_{\chi}((f\otimes\id)(W)(\id\otimes g)(V)).$$
On the other hand, the weak* continuity of one-sided operator multiplication and that of the maps
$\Gamma_W\otimes\Gamma_V'$, $\cl T$ and $f\otimes \id\otimes g$
imply that
\begin{eqnarray*}
& & (f\otimes \id\otimes g)(W_{1,2} (\cl T\circ (\Gamma_W\otimes\Gamma_V')(\chi_\alpha) )V_{2,3})  \to_{\alpha\in \bb{A}}^{w^*}\\
& &
(f\otimes \id\otimes g)(W_{1,2} (\cl T\circ (\Gamma_W\otimes\Gamma_V')(\chi) )V_{2,3}).
\end{eqnarray*}
Identity (\ref{eq_slices}) now follows.
\end{proof}

Let
$$\cl A(W,V) = \overline{[L_f(W) R_g(V) : f,g\in \cl B(H)_*]}^{\|\cdot \|};$$
we call $\cl A(W,V)$ the \emph{reduced operator space} of the pair $(W,V)$.
The assumptions $W\in \cl B(H)\bar\otimes\cl N$ and $V\in \cl N\bar\otimes \cl B(H)$
imply that $\cl A(W,V)$ is a (norm closed) subspace of $\cl N$.

\begin{theorem}\label{th_chargenn}
Let $H$ be a Hilbert space, $\cl M$ and $\cl N$ be von Neumann algebras acting on $H$,
$\chi\in \cl M\otimes_{\sigma\hh}\cl M$,
and $W \in \cl B(H)\bar\otimes\cl N$, $V \in \cl N \bar\otimes \cl B(H)$ be unitary operators.
The following are equivalent:

(i) \ $\Phi_{\chi}(\cl A(W,V))\subseteq \cl N$;

(ii) $\cl T\circ (\Gamma_W\otimes\Gamma'_V)(\chi) \in \cl B(H)\bar\otimes\cl N \bar\otimes\cl B(H)$.
\end{theorem}
\begin{proof}
(i)$\Rightarrow$(ii)
By Lemma \ref{l_ident},
$$(f\otimes \id\otimes g)(W_{1,2} (\cl T\circ (\Gamma_W\otimes\Gamma'_V)(\chi) )V_{2,3}) \in \cl N$$
for all $f,g\in \cl B(H)_*$.
It follows from \cite{tom} that
$$W_{1,2} (\cl T\circ (\Gamma_W\otimes\Gamma'_V)(\chi) )V_{2,3} \in \cl B(H)\bar\otimes \cl N\bar\otimes \cl B(H).$$
Thus,
$$\cl T\circ (\Gamma_W\otimes\Gamma'_V)(\chi) \in \cl B(H)\bar\otimes \cl N\bar\otimes \cl B(H).$$

(ii)$\Rightarrow$(i) is immediate from Lemma \ref{l_ident}.
\end{proof}

%%%%%%%%%%%%%%%%         QUANTUM GROUPS         %%%%%%%%%%%%%%%%%%

\section{Applications to locally compact quantum groups}\label{s:LCQG}

We refer the reader to  \cite{kv, va} (see also \cite{junge-Neufang})
for background on (von Neumann algebraic) locally compact quantum groups.
In the sequel, we recall only those elements of the theory that will be essential for the
statements and the proofs of our results.
A locally compact quantum group is a quadruple
$\G = (L^\infty(\G),\Gamma, \varphi,\psi)$, where $L^\infty(\G)$ is a von
Neumann algebra equipped with a co-associative co-multiplication $\Gamma : L^\infty(\G)\to L^\infty(\G)\bar\otimes L^\infty(\G)$, and
$\varphi$ and $\psi$ are (normal faithful semifinite) left and right Haar weights on
$L^\infty(\G)$, respectively.  The left Haar weight induces an inner product
\begin{equation}\label{inner}
\langle x,y\rangle_\varphi = \varphi(y^*x)
\end{equation}
on the subspace
$\mathfrak R_\varphi=\{x\in L^\infty(\G): \varphi(x^*x)<\infty\}$.
We let $L^2(\G,\varphi)$ denote the completion of $\mathfrak R_\varphi$
with respect to the norm induced by (\ref{inner}).
We define the Hilbert space $L^2(\G,\psi)$ in a similar way.

Let $W\in \cl B(L^2(\G,\varphi)\otimes L^2(\G,\varphi))$
(resp. $V\in \cl B(L^2(\G,\psi)\otimes L^2(\G,\psi))$) be the
left (resp. right) fundamental unitary associated with $\G$
and note (see (\ref{eq_gammas})) that
$$\Gamma = \Gamma_W = \Gamma_V'.$$
Let $L^1(\G)$ be the predual of $L^\infty(\G)$.
The pre-adjoint of $\Gamma$ induces an associative completely contractive multiplication $\ast$ on $L^1(\G)$ by letting
$$ f_1\ast f_2 = (f_1\otimes f_2)\circ \Gamma.$$
The \emph{left regular representation} $\lambda : L^1(\G)\to\cl B(L^2(\G,\varphi))$ is defined by
  \begin{equation}
  \lambda(f) = L_f(W);
  \end{equation}
note that $\lambda$ is an injective completely contractive algebra homomorphism.
We have that $L^\infty(\hat\G) \stackrel{def}{=} \{\lambda(f) : f\in L^1(\G)\}^{''}$
is the von Neumann algebra associated to the dual quantum group $\hat\G$ of $\G$.
If $\Sigma$ is the flip operator on $L^2(\G, \varphi) \otimes L^2(\G, \varphi)$ and $\hat{W} = \Sigma W^* \Sigma$,
its co-multiplication $\hat{\Gamma}$ is given by
\[
\hat{\Gamma}(\hat{x})=\hat{W}^* (1 \otimes \hat{x}) \hat{W}, \quad \quad \hat{x} \in L^\infty(\hat{\G}).
\]
The norm closure
$$C_0(\hat\G) \stackrel{def}{=} \overline{\{\lambda(f) : f\in L^1(\G)\}}$$
is referred to as the \emph{reduced C*-algebra} associated with $\hat\bG$.

Analogously, the \emph{right regular representation} $\rho : L^1(\G) \to \cl B(L^2(\G,\psi))$ is given by
 $$\rho(g) = R_g(V), \ \ \ g\in L^1(\G);$$
note that $\rho$ is an injective completely contractive algebra homomorphism.
We have that
$L^\infty(\hat\G') = \{\rho(g): g\in L^1(\G)\}^{''}$ is the von Neumann algebra associated with
a quantum group denoted $\hat \G'$.
We note that
$L^\infty(\hat\G') = L^\infty(\hat\G)'$,
$W\in L^\infty(\G)\bar\otimes L^\infty(\hat\G)$ and $V\in L^\infty(\hat \G')\bar\otimes L^\infty(\G)$.

Let
$$\Gamma_{\rm op} : L^\infty(\G) \to L^\infty(\G) \bar{\otimes} L^\infty(\G)$$
be the map given by
$\Gamma_{\rm op}(x) = (\sigma\circ \Gamma)(x)$,
where
\begin{equation}\label{eq_flipp}
\sigma : L^\infty(\G) \bar{\otimes} L^\infty(\G) \to L^\infty(\G) \bar{\otimes} L^\infty(\G),  \ \ \
\sigma(T)  =  \Sigma T \Sigma,
\end{equation}
is the flip.
Note that $\Gamma_{\rm op} = \Gamma'_{\hat{W}}$.
It is known that $L^2(\G,\varphi)$ and $L^2(\G, \psi)$ can be canonically identified;
we use $L^2(\G)$ for this Hilbert space in the rest of this paper.

\begin{theorem}\label{t:(m)-iff-VN-invariant}
Let $\G$ be a   locally compact  quantum group and $\chi\in L^\infty(\G)\otimes_{\sigma\hh} L^\infty(\G)$.
The following are equivalent:

(i) \ $\Phi_{\chi}(C_0(\hat\bG))\subseteq L^{\infty}(\hat\bG)$;

(ii)
$\mathcal T\circ(\Gamma\otimes\Gamma_{\rm op})(\chi)\in L^\infty(\G)\bar\otimes 1 \bar\otimes L^\infty(\G)$.
\end{theorem}
\begin{proof}
Note that, if $f,g\in L^1(\G)$ then
$(f\otimes \id)(W) = \lambda(f)$ and $(\id \otimes g)(\hat W) = \lambda(g)$.
Thus,
\[
\cl A(W,\hat W) = \overline{[\lambda(f)\lambda(g) : f,g\in L^1(\G)]}^{\| \cdot\|}
\]
and hence, by equation  \cite[2.5]{hnr},
$\cl A_{W,\hat W} = C_0(\hat\G)$.

By Theorem \ref{th_chargenn}, condition (i) is equivalent to the condition
$$\mathcal T\circ(\Gamma\otimes\Gamma_{\rm op})(\chi)\in \cl B(L^2(\G))\bar\otimes L^{\infty}(\hat\G) \bar\otimes \cl B(L^2(\G)).$$
On the other hand, by Lemma \ref{l:T},
$$\mathcal T\circ(\Gamma\otimes\Gamma_{\rm op})(\chi)\in L^{\infty}(\G)\bar\otimes L^{\infty}(\G) \bar\otimes L^{\infty}(\G).$$
Since $L^{\infty}(\hat\G)\cap L^{\infty}(\G) = \bb{C}1$, we conclude by \cite{tom}
and Theorem \ref{th_chargenn} that (i) and (ii) are equivalent.
\end{proof}

\begin{corollary}\label{c:(m)-iff-VN-invariant}
Let $\G$ be a   locally compact  quantum group and suppose that
$\chi\in L^\infty(\G)\otimes_{\eh} L^\infty(\G)$.
The following are equivalent:

(i) \ $\Phi_{\chi}(L^{\infty}(\hat\bG))\subseteq L^{\infty}(\hat\bG)$;

(ii) $\mathcal T\circ(\Gamma\otimes\Gamma_{\rm op})(\chi)\in L^\infty(\G)\bar\otimes 1 \bar\otimes L^\infty(\G)$.
\end{corollary}
\begin{proof}
The statement follows from Theorem \ref{t:(m)-iff-VN-invariant},
the weak* continuity of $\Phi_{\chi}$ and the fact that $C_0(\hat\bG)$
is weak* dense in $L^\infty(\hat\G)$.
\end{proof}

%%%%%%%%%%%%%%%%      CO-COMMUTATIVE CASE      %%%%%%%%%%%%%%%%%%

\section{The co-commutative case}\label{ss_coc}

In this section, we specialise the results from Section \ref{s:LCQG} to co-commutative
locally compact quantum groups $\G$ \cite{kv}.
Let $G$ be a locally compact group equipped with left Haar measure.
We thus consider the case where $L^{\infty}(\G)$ coincides with the von Neumann algebra $\vn(G)$ of $G$,
acting on the Hilbert space $L^2(G)$.
We have that $C_0(\G)$ is equal to the reduced groups C*-algebra $C_r^*(G)$ of $G$.
We denote by $\lambda$ (resp. $\rho$) the left (resp. right) regular representation of $G$ on $L^2(G)$.
Using the canonical identification
\begin{equation}\label{eq_vnvn}
\vn(G)\bar\otimes \vn(G) \equiv \vn(G\times G),
\end{equation}
we have that the co-multiplication $\Gamma : \vn(G) \to \vn(G\times G)$ is given by
\[
\Gamma(\lambda(s)) = \lambda(s)\otimes\lambda(s), \ \ \ s\in G.
\]
In this case, $L^1(\G)$ coincides with the Fourier algebra $A(G)$ of $G$.
We refer the reader to \cite{eymard} for the basic concepts and results from Abstract Harmonic Analysis that will be used in this section,
but note here that $A(G)$ consists of (complex valued) continuous functions on $G$ vanishing at infinity,
the operation is point-wise, and its Gelfand spectrum can be homeomorphically identified with $G$.

Note that, in the case under consideration, $L^{\infty}(\hat\G)$ coincides with $L^{\infty}(G)$, 
which we hereafter view as the von Neumann algebra
of operators on $L^2(G)$ of multiplication by the corresponding essentially bounded functions.
Under this identification, 
$C_0(\hat\G)$ coincides with the subalgebra $C_0(G)$ of all continuous functions vanishing at infinity.

Equation (\ref{eq_vnvn}) allows us to identify the predual $A(G\times G)$ of $\vn(G\times G)$ with the operator projective
tensor product $A(G)\hat{\otimes} A(G)$.
Note that the pre-adjoint of $\Gamma$ is the associative multiplication map
$\m_{A} : A(G)\hat\otimes A(G)\to A(G)$ given by
$\m_{A}(f\otimes g) = fg$.

We set
$$A_{\hh}(G) = A(G)\otimes_{\hh} A(G) \ \mbox{ and } \ A_{\eh}(G) = A(G)\otimes_{\eh} A(G);$$
note that there is a completely isometric inclusion $A_{\hh}(G)\subseteq A_{\eh}(G)$ \cite{er}.
We have completely isometric identifications
$$A_{\hh}(G)^* = \vn(G)\otimes_{\eh} \vn(G) \ \mbox{ and } \ A_{\eh}(G)^* = \vn(G)\otimes_{\sigma \hh} \vn(G).$$
In \cite{bi}, $A_{\eh}(G)$ (resp. $A_{\hh}(G)$) was identified with a completely contractive Banach algebra of
separately (resp. jointly) continuous complex-valued functions on $G\times G$, equipped with pointwise multiplication;
specifically, the function, corresponding to an element $v\in A_{\eh}(G)$
is (denoted in the same fashion and) given by
$$v(s,t) = \langle \lambda_s\otimes \lambda_t, v\rangle, \ \ \ s,t\in G.$$
Note that $A_{\hh}(G)$ is a regular (semi-simple) Banach algebra whose Gelfand spectrum
can be homeomorphically identified with $G\times G$.

Recall that the multiplication map $\m : \vn(G)\times \vn(G) \to \vn(G)$ (given by $\m(S,T) = ST$)
extends uniquely to a weak* continuous completely contractive map (denoted in the same fashion)
$\m : \vn(G)\otimes_{\rm \sigma h} \vn(G)\to \vn(G)$.
If $\m_*:A(G)\to A(G)\otimes_{\eh}A(G)$ is the pre-adjoint of $\m$ then
$$\m_*(f)(s,t) = f(st), \ \ \ f\in A(G), s,t\in G.$$

M. Daws has shown \cite{daws} that, if $f\in A(G)$ then $\m_*(f)$ belongs to the
algebra $M^{\rm{cb}} A_{\hh}(G)$ of completely bounded multipliers \cite{bi} of $A_{\hh}(G)$; in particular,
$$\m_*(f) u \in A_{\hh}(G), \ \mbox{ whenever } u \in A_{\hh}(G).$$
Moreover,
\begin{equation}\label{eq:m-*}
\m_* : A(G) \to M^{\rm{cb}} A_{\hh}(G)
\end{equation}
is completely contractive (we refer the reader to \cite[Theorem~9.2]{daws} and the remark after its proof).

Theorem \ref{t:(m)-iff-VN-invariant} specialises in the case under consideration to the following result,
providing a characterisation of the completely bounded maps on $\cl B(L^2(G))$ that
are not necessarily normal and which map $C_0(G)$ into $L^{\infty}(G)$.

\begin{corollary}\label{c_VN-invariant}
Let $G$ be a   locally compact group and $\chi\in \vn(G)\otimes_{\sigma\hh} \vn(G)$.
The following are equivalent:

(i) \ $\Phi_{\chi}(C_0(G))\subseteq L^{\infty}(G)$;

(ii)
$\mathcal T\circ(\Gamma\otimes\Gamma_{\rm op})(\chi)\in \vn(G)\bar\otimes 1 \bar\otimes \vn(G)$.
\end{corollary}

If $\chi\in \vn(G)\otimes_{\sigma \hh}\vn(G)$ and $v\in A_{\eh}(G)$, we let $v\cdot \chi$ be the (unique) element of
$\vn(G)\otimes_{\sigma \hh}\vn(G)$ such that
$$\langle v\cdot \chi, w\rangle = \langle \chi, vw\rangle, \ \ \ w\in A_{\eh}(G).$$
If $\chi\in \vn(G)\otimes_{\eh} \vn(G)$ and $v\in A_{\hh}(G)$ then
$v\cdot \chi$ denotes the similarly defined element of $\vn(G)\otimes_{\eh} \vn(G)$.
If $\chi\in \vn(G)\otimes_{\eh}\vn(G)$, we
denote by $\supp_{\eh}(\chi)$ the support of $\chi$ when the latter
is viewed as a functional on $A_{\hh}(G)$;
thus,
$$\supp\mbox{}_{\eh}(\chi) = \{z\in G\times G : v\in A_{\hh}(G), v(z) \neq 0 \  \Longrightarrow \ v\cdot \chi \neq 0\}$$
(see \cite{eymard} and \cite{katz}).
Let $\frak{p} : \vn(G)\otimes_{\sigma \hh}\vn(G)\to \vn(G)\otimes_{\eh}\vn(G)$
be the natural projection given by $\frak{p}(\chi) = \chi|_{A_{\hh}(G)}$
(or, equivalently, the dual of the inclusion map $A_{\hh}(G)\to A_{\eh}(G)$),
and define $\supp_{\sigma}(\chi) = \supp_{\eh}(\frak{p}(\chi))$.

If $E\subseteq G \times G$ is a closed set, let
$$J_{\hh}(E) = \{u\in A_{\hh}(G) : u \mbox{ has compact support, disjoint from } E\overline{\}}^{\|\cdot\|},$$
and write $J_{\hh}(E)^{\perp,\eh}$ (resp. $J_{\hh}(E)^{\perp,\sigma}$) for its annihilator in
$\vn(G)\otimes_{\eh}\vn(G)$ (resp. $\vn(G)\otimes_{\sigma\hh}\vn(G)$).
It is well-known that
\begin{equation}\label{eq_hhan}
J_{\hh}(E)^{\perp,\eh} = \{\chi \in \vn(G)\otimes_{\eh}\vn(G) : \supp\mbox{}_{\eh}(\chi)\subseteq E\};
\end{equation}
the following, similar, description of $J_{\hh}(E)^{\perp,\sigma}$ is straightforward from the definitions.

\begin{remark}\label{l_supsig}
Let $E$ be a closed subset of $G\times G$. Then
$$J_{\hh}(E)^{\perp,\sigma} = \{\chi \in \vn(G)\otimes_{\sigma\hh}\vn(G) : \supp\mbox{}_{\sigma}(\chi)\subseteq E\}.$$
\end{remark}

Let
$$\nabla = \{(s,s^{-1}) : s\in G\}$$
be the \emph{anti-diagonal} of $G\times G$.
Recall the flip $\sigma : \vn(G\times G) \to \vn(G\times G)$ defined in (\ref{eq_flipp}) and note that
\[
(\id \otimes \sigma)\big(a \otimes b \otimes c\big) = a \otimes c \otimes b, \ \ \ a,b,c  \in \vn(G).
\]
In condition (iii) of the next theorem, we identify $\vn(G)\otimes_{\eh}\vn(G)$ with a subspace of $\vn(G\times G)$
(see \cite[Corollary 3.8]{blecher}).

\begin{theorem}\label{delta_prop}
Let $\chi\in \vn(G)\otimes_{\sigma\hh} \vn(G)$.
The following are equivalent:

(i) \ $\mathcal T\circ(\Gamma\otimes\Gamma) (\chi) \in \vn(G) \bar\otimes 1 \bar\otimes \vn(G)$;

(ii) $\supp_{\sigma}(\chi) \subseteq  \nabla$.

\noindent If $\chi\in \vn(G)\otimes_{\eh} \vn(G)$ then these conditions are also equivalent to

(iii) \ $\mathcal T\circ(\Gamma\otimes\Gamma) (\chi) = (\id \otimes \sigma)\big(\chi\otimes 1 \big)$.
\end{theorem}

\begin{proof}
Assume that $\chi = \lambda_s\otimes \lambda_t$, where $s,t\in G$.
For $f,g,\omega \in A(G)$, using Lemma~\ref{l:T}, we obtain
\begin{eqnarray*}
\langle \mathcal T\circ(\Gamma\otimes\Gamma)(\chi), f\otimes \omega \otimes g\rangle
& = &
\langle \mathcal T(\lambda_s\otimes\lambda_s\otimes\lambda_t \otimes \lambda_t), f\otimes \omega \otimes g \rangle\\
& = &
\langle \lambda_s\otimes\lambda_{st}\otimes\lambda_t, f\otimes \omega \otimes g \rangle\\
& = &
f(s) \omega(st) g(t) =  f(s) \m_*(\omega)(s,t) g(t)\\
& = &
\langle \lambda_s\otimes\lambda_t, \m_*(\omega)(f\otimes g)\rangle\\
& = &
\langle \chi, \m_*(\omega)(f\otimes g)\rangle\\
& = &
\langle \m_*(\omega) \cdot \chi, f\otimes g\rangle.
\end{eqnarray*}

It is easy to see that the linear span of $\{\lambda_s\otimes \lambda_t : s,t \in G\}$ is weak* dense in $\vn(G)\otimes_{\sigma\hh}\vn(G)$.
Further, since the map $\psi \to \m_*(\omega) \cdot \psi$ on  $\vn(G)\otimes_{\sigma\hh}\vn(G)$
is weak* continuous,
we have that
\begin{equation}\label{eq_fomegag}
\langle \mathcal T\circ(\Gamma\otimes\Gamma)(\chi), f\otimes \omega \otimes g\rangle
= \langle \m_*(\omega) \cdot \chi, f\otimes g\rangle,
\end{equation}
for all $\chi\in \vn(G)\otimes_{\sigma\hh}\vn(G)$ and all $f,g,\omega\in A(G)$.

\medskip

(i)$\Rightarrow$(ii)
By assumption, there exists $\nph\in \vn(G)\bar\otimes\vn(G)$ such that
$$\mathcal T\circ(\Gamma\otimes\Gamma) (\chi) = (\id \otimes \sigma)(\nph\otimes 1).$$
Thus, if $f,g,\omega\in A(G)$ then (\ref{eq_fomegag}) implies
\begin{equation}\label{eq_mstar}
\langle \m_*(\omega) \cdot \chi, f\otimes g\rangle = \langle (\id \otimes \sigma)(\nph\otimes 1), f\otimes \omega \otimes g\rangle
= \langle \omega(e) \nph, f\otimes g\rangle.
\end{equation}

Fix $(s,t)\in G\times G$ such that $st\neq e$. Let $f,g\in A(G)$ be such that
$f(s) g(t) \neq 0$. Choose $\omega_1, \omega_2\in A(G)$ such that
$\omega_1(e) = \omega_2(e)$ but $\omega_1(st) \neq \omega_2(st)$.
Thus, 
\begin{equation}\label{eq_nonz}
(\m_*(\omega_1) - \m_*(\omega_2))(f\otimes g)(s,t)\neq 0.
\end{equation}
By (\ref{eq_mstar}), for all $f_0, g_0\in A(G)$, we have
\begin{eqnarray*}
& &
\langle (\m_*(\omega_1) - \m_*(\omega_2))(f\otimes g) \cdot \chi, f_0\otimes g_0\rangle\\
& = &
\langle (\m_*(\omega_1) - \m_*(\omega_2)) \cdot \chi, (f f_0) \otimes (g g_0) \rangle\\
& = &
\langle (\omega_1(e) - \omega_2(e)) \nph, (f f_0) \otimes (g g_0) \rangle = 0.
\end{eqnarray*}
We have that $(\m_*(\omega_1) - \m_*(\omega_2))(f\otimes g) \in A_{\hh}(G)$ and hence
$$\langle (\m_*(\omega_1) - \m_*(\omega_2))(f\otimes g) \cdot \frak{p}(\chi), f_0\otimes g_0\rangle = 0.$$
Since $A(G)\odot A(G)$ is dense in $A_{\hh}(G)$, we conclude that
$$(\m_*(\omega_1) - \m_*(\omega_2))(f\otimes g) \cdot \frak{p}(\chi) = 0.$$
In view of (\ref{eq_nonz}), 
this shows that $(s,t)\not\in \supp_{\eh}(\frak{p}(\chi))$, and thus $\supp_{\sigma}(\chi)\subseteq \nabla$.

(ii)$\Rightarrow$(i)
Assume that $\supp_\sigma(\chi) \subseteq \nabla$.  By Remark \ref{l_supsig}, 
\[
\langle \frak{p}(\chi),u \rangle = 0 \text{ for any }u \in \ahg\cap \ C_c(G\times G) \mbox{ with } \supp(u) \cap \nabla = \emptyset.
\]
Let $u \in \ahg\cap \ C_c(G\times G)$.  Choose $h \in A(G)$
that takes the value 1 on the (compact) set $\{st : (s,t)\in\supp u\} \cup \{e\}$.
Since $\{e\}$ is a set of synthesis for $A(G)$,
for every $\omega \in A(G)$ with $\omega(e) = 1$, there exists $h_n \in A(G)\cap C_c(G)$, $n\in\mathbb N$, vanishing on
a neighbourhood of $\{e\}$, such that
\[
\|\omega - h - h_n\|_{A(G)}\to_{n\to\infty} 0.
\]
Since $\m_*$ is a complete contraction from $A(G)$ into $M^{\cb}A_{\hh}(G)$ \cite{daws}, we have
that $\m_*(\omega - h - h_n) \in M^{\rm cb}\ahg$ and
\[
\|\m_*(\omega - h - h_n)u\|_{\hh} \leq \|\omega - h - h_n\|_{A(G)}\| u \|_{\hh}\to_{n\to\infty} 0.
\]
Since $\supp (\m_*(h_n) u )\cap \nabla = \emptyset$, we have that
$\langle \frak{p}(\chi), \m_*(h_n)u \rangle = 0$. On the other hand, $\m_*(h) u = u$ and hence
\begin{eqnarray*}
\langle \frak{p}(\chi), (\m_*(\omega) - 1)u \rangle
& = & \langle\frak{p}(\chi), (\m_*(\omega)-\m_*(h))u\rangle\\
& = & \langle\frak{p}(\chi), (\m_*(\omega-h- h_n)u \rangle \to_{n\to\infty} 0,
\end{eqnarray*}
i.e. $\langle\frak{p}(\chi), (\m_*(\omega) - 1)u \rangle = 0$.
In particular,
since $\m_*(\omega) (f\otimes g)\in A_{\hh}(G)$, we obtain
\[
\langle \chi,\m_*(\omega) (f\otimes g)\rangle =
\langle\frak{p}(\chi),\m_*(\omega) (f\otimes g)\rangle = \langle\frak{p}(\chi),f\otimes g\rangle
= \langle \chi,f\otimes g\rangle
\]
for all $f$, $g\in A(G)\cap C_c(G)$.
Using (\ref{eq_fomegag}), we obtain that, if $f,g,\omega\in A(G)$ then
\begin{eqnarray*}
\langle (f\otimes \id \otimes g)(\mathcal T\circ(\Gamma\otimes\Gamma)(\chi)), \omega \rangle
& = &
\langle \mathcal T\circ(\Gamma\otimes\Gamma)(\chi), f\otimes \omega \otimes g\rangle\\
& = &
\omega(e) \langle \chi,f\otimes g\rangle
= \langle \langle \chi,f\otimes g\rangle 1,\omega \rangle.
\end{eqnarray*}
Thus,
\begin{equation}\label{eq_fgI}
(f\otimes \id \otimes g)(\mathcal T\circ(\Gamma\otimes\Gamma)(\chi)) = \langle \chi,f\otimes g\rangle 1, \ \ \ f,g\in A(G).
\end{equation}
It follows that $\mathcal T\circ(\Gamma\otimes\Gamma)(\chi)) \in \vn(G) \bar\otimes 1 \bar\otimes \vn(G)$.

Now suppose that $\chi\in \vn(G)\otimes_{\eh} \vn(G)$.
In this case, (\ref{eq_fgI}) can be rewritten as
$$\langle \mathcal T\circ(\Gamma\otimes\Gamma)(\chi), f\otimes \omega \otimes g\rangle =
\langle (\id \otimes \sigma)\big(\chi\otimes 1\big), f\otimes \omega \otimes g\rangle,$$
completing the proof.
\end{proof}

In the next corollary, $M(G)$ denotes the Banach algebra of all complex Borel measures on $G$.
The equivalence (i)$\Leftrightarrow$(ii) extends \cite[Theorem 6.10]{bi}.
The integral in (iii) is understood in the weak sense.

\begin{corollary}
Let $\chi \in \vn(G)\otimes_{\eh} \vn(G)$. The following are equivalent

(i) \ \ $\Phi_{\chi} \in {\rm CB}_{\vn(G)'}^{\sigma, L^\infty(G)}(\cl B(L^2(G)))$;

(ii) \ $\supp\chi\subseteq \nabla$;

(iii) there exists $\mu\in M(G)$ such that $\Phi_{\chi}(x) = \int_G\lambda_sx\lambda_s^*d\mu(s)$, $x\in \cl B(L^2(G))$.
\end{corollary}

\begin{proof}
The equivalence (i)$\Leftrightarrow$(ii) follows from
Corollary \ref{c:(m)-iff-VN-invariant} and Theorem \ref{delta_prop}, while
the equivalence (i)$\Leftrightarrow$(iii) was proved in \cite{neufang}.
\end{proof}

%%%%%%%%%%%%%%%%        COMMUTATIVE CASE        %%%%%%%%%%%%%%%%%%

\section{The commutative case}\label{ss_com}

In this section, we specialise the results from Section \ref{s:LCQG} to commutative locally compact quantum groups $\G$.
Thus, fixing a second countable locally compact group $G$, we have that $L^{\infty}(\G) = L^{\infty}(G)$,
$C_0(\G) = C_0(G)$, $L^{\infty}(\hat\G) = \vn(G)$ and $C_0(\hat\G) = C_r^*(G)$.
We identify the tensor product $L^{\infty}(G)\bar\otimes L^{\infty}(G)$ with the von Neumann algebra $L^{\infty}(G\times G)$
in the canonical way.
%thus, the function corresponding to the elementary tensor $a\otimes b$, where $a,b\in L^{\infty}(G)$,
%is (denoted by the same symbol and) given by $(a\otimes b)(s,t) = a(s) b(t)$.
It is readily verified that the flip $\sigma : L^{\infty}(G\times G) \to L^{\infty}(G\times G)$ is given by
\begin{equation}\label{eq_sth}
\sigma(h)(s,t) = h(t,s), \mbox{ for almost all } (s,t)\in G\times G, \ \ h\in L^{\infty}(G\times G),
\end{equation}
while the co-multiplication $\Gamma : L^\infty(G)\to L^{\infty}(G\times G)$ -- by
\begin{equation}\label{eq_gamma}
\Gamma (a)(s,t) = a(st), \mbox{ for almost all } (s,t)\in G\times G, \ \ a\in L^{\infty}(G).
\end{equation}
It follows that
$\Gamma_{\rm op} : L^\infty(G)\to L^\infty(G\times G)$ is given by
\begin{equation}\label{eq_gammaop}
\Gamma_{\rm op}(a)(s,t) = a(ts),  \mbox{ for almost all } (s,t)\in G\times G,  \ \ a\in L^{\infty}(G).
\end{equation}
Note that the predual of $L^{\infty}(G)$ is (completely) isometric to $L^1(G)$ and the pre-adjoint of $\Gamma$
is the usual convolution product:
$$(f\ast g)(t) = \int f(s)g(s^{-1}t)dt, \ \ t\in G, \ f,g\in L^1(G).$$

Theorem \ref{t:(m)-iff-VN-invariant} specialises in the case under consideration to the following result,
providing a characterisation of the completely bounded maps on $\cl B(L^2(G))$ that
are not necessarily normal and which map $C_r^*(G)$ into $\vn(G)$.

\begin{corollary}\label{c_com}
Let $\chi\in L^\infty(G)\otimes_{\sigma\hh} L^\infty(G)$.
The following are equivalent:

(i) \ $\Phi_{\chi}(C_r^*(G))\subseteq \vn(G)$;

(ii)  $\mathcal T\circ(\Gamma\otimes\Gamma_{\rm op})(\chi) \in L^{\infty}(G) \bar\otimes 1 \bar\otimes L^{\infty}(G)$.
\end{corollary}

In the rest of this subsection, we restrict our attention to the case where the
elements $\chi$ from Corollary \ref{c_com} belong to the extended Haagerup tensor product $L^\infty(G)\otimes_{\eh} L^\infty(G)$.
Those are precisely the elements $\chi \in L^\infty(G)\otimes_{\sigma\hh} L^\infty(G)$,
for which the corresponding map $\Phi_{\chi}$ is ($L^{\infty}(G)$-bimodular and) weak* continuous,
and are widely known in the literature as \emph{Schur multipliers}.

Set $V^\infty(G) = L^\infty(G)\otimes_{\eh}L^\infty (G)$ and note that
$V^\infty(G) = (L^1(G)\otimes_{\hh} L^1(G))^*$, up to a complete isometry.
Let $\chi \in V^{\infty}(G)$
and assume that $\chi \sim \sum_{i=1}^{\infty} a_i \otimes b_i$, where
$(a_i)_{i\in\mathbb N}$ and $(b_i)_{i\in \mathbb N}$ in $L^\infty(G)$ are sequences
such that
$$\esssup_{s\in G} \sum_{i = 1}^{\infty} |a_i(s)|^2 < \infty \ \mbox{ and } \  \esssup_{t\in G} \sum_{i = 1}^{\infty} |b_i(t)|^2 < \infty$$
\cite{blecher}.
We identify $\chi$ with the essentially bounded function (denoted in the same way)
$\chi : G\times G \to \bb{C}$, given by
\begin{equation}\label{eq_aibi}
\chi(s,t) = \sum_{i=1}^{\infty} a_i(s)b_i(t), \ \ \ s,t\in G,
\end{equation}
and thus consider $V^{\infty}(G)$ as a subspace of $L^{\infty}(G\times G)$.

Given a function $\nph : G\times G\to \bb{C}$ and $r\in G$, let $\nph_r : G\times G \to \bb{C}$ be the function given by
$\nph_r(s,t) = \nph(sr,tr)$.

\begin{lemma}\label{l_trans}
Let $\chi\in V^{\infty}(G)$ and $r\in G$. Then $\chi_r\in V^{\infty}(G)$ and
$$\Phi_{\chi_r}(T) = \rho_r\Phi_{\chi}(\rho_r^* T \rho_r) \rho_r^*, \ \ \ T\in \cl B(L^2(G)).$$
\end{lemma}
\begin{proof}
For $a\in L^{\infty}(G)$ let $a_r\in L^{\infty}(G)$ be given by
$a_r(s) = a(sr)$, $s\in G$. A direct verification shows that $\rho_r M_a \rho_r^* = M_{a_r}$.
Clearly, if $\chi \sim \sum_{i=1}^{\infty} a_i\otimes b_i$
then $\chi_r(s,t) = \sum_{i=1}^\infty (a_i)_r(s)(b_i)_r(t)$, for almost all $(s,t)$. 
Now, if $T\in \cl B(L^2(G))$ then
\begin{eqnarray*}
\rho_r\Phi_{\chi}(\rho_r^* T \rho_r) \rho_r^*
& = &
\sum_{i=1}^{\infty} (\rho_r M_{a_i} \rho_r^*) T (\rho_r M_{b_i} \rho_r^*)\\
& = &
\sum_{i=1}^{\infty} M_{(a_i)_{r}}  T M_{(b_i)_{r}} = \Phi_{\chi_r}(T).
\end{eqnarray*}
\end{proof}

Recall that, if $X$ is a measure space and $X\times X$ is equipped with the product measure,
a measurable subset $E\subseteq X\times X$ is called \emph{marginally null} \cite{a} if
there exists a null set $M\subseteq X$ such that $E\subseteq (M\times X)\cup (X\times M)$.
If $w_1,w_2 : X\times X\to \bb{C}$ are measurable functions, we say that $w_1$ and $w_2$
are equal marginally almost everywhere (m.a.e.) if
the set $\{(x,y) : w_1(x,y) \neq w_2(x,y)\}$ is marginally null.
We note that if $\chi \in V^\infty(G)$ then the function $\chi$ is well-defined up to a marginally null subset of $G\times G$
and that it completely determines the corresponding map $\Phi_{\chi}$ (see \cite{todorov}).

In the proof of the following lemma, given an element of
$L^\infty(G\times G)$, we choose any representative of its equivalence class,
and denote it still by the same symbol.
%with a function on $G \times G$ in its equivalence class. In particular, if $c \geq 0$, as an element in $L^\infty(G \times G)$, the chosen representative is a function so that $c(s,t) \geq 0$ for all $s,t \in G$.

\begin{lemma}\label{l_tildechi}
Let $\chi\in V^\infty(G\times G)$. The function
$\tilde\chi : G\times G\times G\to \bb{C}$, given by
$$\tilde\chi(s,r,t) =  \chi(s,r,r,t), \ \ \ s,r,t\in G,$$
is a well-defined element of $L^\infty(G\times G\times G)$.
\end{lemma}
\begin{proof}
For a measurable function $c : \cl Z\to \bb{R}$, where $\cl Z$ is a measure space, 
and a positive real number $\delta$, let
$E_c(\delta) = \{z\in \cl Z : c(z) > \delta\}$.

Suppose that $\chi = a\otimes b$, where $a,b\in L^{\infty}(G\times G)$.
In order to show that $\tilde\chi$ is measurable in this case, it suffices to assume that
$a \geq 0$ and $b\geq 0$.
Then
$$E_{\tilde{\chi}}(\delta) = \cup_{\lambda\in \bb{Q}^+} \{(s,r,t) : (s,r)\in E_a(\lambda) \mbox{ and } (r,t)\in E_b(\delta/\lambda)\}.$$
If $F\subseteq G\times G$ is a set, setting $F^r = \{s\in G : (s,r)\in F\}$ and
$F_r = \{t\in G : (r,t)\in F\}$, we have
$$\{(s,t) : (s,r,t) \in E_{\tilde{\chi}}(\delta)\}
= \cup_{\lambda\in \bb{Q}^+} (E_a(\lambda)^r \times G) \cap (G \times E_b(\delta/\lambda)_r),$$
for each $r\in G$.
It follows that $\tilde\chi$ is a measurable function, when $G\times G \times G$ is equipped with product measure.

For a general $\chi\in V^\infty(G\times G)$, write
$\chi \sim \sum_{i = 1}^{\infty} a_i\otimes b_i$ and let $M$ be a null subset of $G\times G$ such that
\[
\sum_{i=1}^\infty |a_i(s,t)|^2 < C \mbox{ and } \sum_{i=1}^\infty|b_i(s,t)|^2 < C  \text{ whenever } (s,t) \not\in M.
\]
Set $\chi_n = \sum_{i = 1}^{n} a_i\otimes b_i$, $n\in \bb{N}$.
It is clear that
$$\tilde\chi(s,r,t) =  \lim_{n\to \infty} \tilde\chi_n (s,r,t), \ \ \ (s,r,t) \not\in (M\times G) \cup (G\times M).$$
It follows from the previous paragraph that $\tilde\chi$ is measurable.

Let $\sum_{i = 1}^{\infty} a_i\otimes b_i$ and $\sum_{j = 1}^{\infty} c_j\otimes d_j$ be two w*-representations of the
same element $\chi \in V^\infty(G\times G)$.
Then the corresponding functions on $G\times G\times G\times G$, let us call them temporarily $\chi_1$ and $\chi_2$,
are equal marginally almost everywhere on $(G\times G)\times (G\times G)$ (see \cite{todorov} for details).
Let $M\subseteq G\times G$ be a null set such that
$\chi_1(s,r,p,t) = \chi_2(s,r,p,t)$ whenever
$((s,r),(p,t))\not\in (M\times (G\times G)) \cup ((G\times G)\times M)$.
Then $\chi_1(s,r,r,t) = \chi_2(s,r,r,t)$ whenever $(s,r)\not\in M$ and $(r,t)\not\in M$,
that is, whenever $(s,r,t)\not\in (M\times G)\cup (G\times M)$.
It follows that the function $\tilde\chi$ is a well-defined element of $L^{\infty}(G\times G \times G)$.
\end{proof}

\begin{lemma}\label{l:T-in-commutative}
Let $G$ be a second countable locally compact group and $\chi\in V^\infty(G\times G)$.
Then $\mathcal T(\chi) = \tilde\chi$.
\end{lemma}
\begin{proof}
For $\omega$, $f$, $g\in L^1(G)$, we have
\[
\langle \mathcal T(\chi), f\otimes \omega \otimes h\rangle = \langle \m((L_f\otimes R_g)(\chi)), \omega\rangle.
\]
Suppose that
$(a_i)_{i\in\mathbb N}$, $(b_i)_{i\in \Bbb{N}}\in L^\infty(G\times G)$ and $C > 0$ are such that
\[
\sum_{i=1}^\infty |a_i(s,t)|^2 < C \mbox{ and }  \sum_{i=1}^\infty|b_i(s,t)|^2 < C \text{ for almost all } (s,t)
\]
and
$\chi \sim \sum_{i = 1}^{\infty} a_i\otimes b_i$.

Let $\chi_n = \sum_{i=1}^n a_i\otimes b_i$, $n\in \bb{N}$.
%Then
%\[
%\m(L_f\otimes L_g)(\chi_n) = \m\left(\sum_{i=1}^n L_f(a_i)\otimes R_g(b_i)\right) = \sum_{i=1}^n L_f(a_i)R_g(b_i).
%\]
It is easy to see that
$$L_f (a_i)(r) = \int_G a_i(s,r)f(s) ds \ \mbox{ and } \ R_g(b_i)(r) = \int_G b_i(r,t)g(t) dt$$
and hence
\begin{equation}\label{eqq}
\langle \m((L_f\otimes R_g)(\chi_n)),\omega \rangle = \int_G\int_{G\times G} \sum_{i=1}^n a_i(s,r)b_i(r,t)f(s)\omega(r)g(t)ds\; dt\; dr.
\end{equation}
By \cite[p. 147]{er},
$\chi_n \to_{n\to \infty} \chi$ in the weak* topology of $L^{\infty}(G \times G) \otimes_{\sigma\hh} L^\infty(G\times G)$;
hence the left hand side of (\ref{eqq}) converges to $\langle \m((L_f\otimes R_g)(\chi)),\omega\rangle$.

On the other hand,
$\chi_n\to_{n\to \infty} \chi$ marginally almost everywhere on $(G\times G)\times(G\times G)$.
Similarly to the proof of Lemma \ref{l_tildechi}, we see that
$\tilde{\chi}_n \to_{n\to \infty} \tilde{\chi}$ almost everywhere on $G\times G\times G$.
Moreover,
\[
|\tilde{\chi}_n(s,r,t)|\leq \left(\sum_{i=1}^n |a_i(s,r)|^2\right)^{1/2} \left(\sum_{i=1}^n|b_i(r,t)|^2\right)^{1/2} < C,
\]
for almost all $(s,r,t)\in G\times G \times G$,
and hence Lebesgue's Dominated Convergence Theorem implies that the right hand side of (\ref{eqq}) converges to
 \[
 \int_G \int_{G\times G} \tilde{\chi} (s,r,t)f(s) \omega(r) g(t) dsdtdr.
 \]
Thus,
\[
\langle \mathcal T(\chi), f\otimes \omega \otimes g\rangle = \langle  \tilde{\chi}, f\otimes \omega \otimes g\rangle,
\]
and the proof is complete.
\end{proof}

\begin{lemma}\label{l_gamma}
Let $\chi\in V^{\infty}(G)$ and
$\hat{\chi} : G\times G \times G \times G \to \bb{C}$ be the function given by
$\hat{\chi}(s,r,p,t) = \chi(sr,tp)$.
Then $\hat{\chi} \in V^{\infty}(G\times G)$ and 
\begin{equation}\label{eq_maehat}
(\Gamma\otimes \Gamma_{\rm op})(\chi) = \hat{\chi}, \ \ \mbox{ m.a.e. on } (G\times G)\times (G\times G).
\end{equation}
\end{lemma}
\begin{proof}
Suppose that
$\chi\sim \sum_{i=1}^{\infty} a_i\otimes b_i$ is a w*-representation of $\chi$.
By the functoriality of the weak* Haagerup tensor product,
$\Gamma \otimes \Gamma_{\rm op}$ is a well-defined map from $V^{\infty}(G)$ into $V^{\infty}(G\times G)$ and
$$(\Gamma\otimes\Gamma_{\rm op})(\chi) \sim \sum_{i=1}^{\infty} \Gamma(a_i)\otimes \Gamma_{\rm op}(b_i)$$
(see \cite[p 136]{blecher}).
However,
$$\sum_{i=1}^{n} \Gamma(a_i)\otimes \Gamma_{\rm op}(b_i) \to_{n\to\infty} \hat{\chi}, \ \ \mbox{ m.a.e. on }( G\times G) \times (G \times G).$$
In fact, there is a null set $M\subseteq G$ such that
$$\sum_{i=1}^n a_i(sr)b_i(tp) \to_{n\to\infty} \sum_{i=1}^\infty a_i(sr)b_i(tp)$$
whenever $(sr,tp)\notin (M\times G)\cup (G\times M)$. Letting $M^\sharp=\{(s,r): sr\in M\}$ and
 $M^\flat = \{(p,t): tp\in M\}$, we have the convergence whenever
 $(s,r,p,t)\notin (M^\sharp\times (G\times G))\cup (G\times G)\times M^\flat$. 
 We finally note that $M^\sharp$ and $M^\flat$ have measure zero.
The conclusion follows.
\end{proof}

Set
$$V^\infty_{\rm inv}(G) = \{\chi \in V^\infty(G) :  \chi_r = \chi  \mbox{ a.e., for all } r\in G\}$$
(see \cite{todorov} and \cite{spronk}).

\begin{theorem}\label{th_chi1}
Let $G$ be a second countable locally compact group and  $\chi\in V^\infty(G)$. The following are equivalent:

(i) \ \ $\Phi_{\chi}(\vn(G))\subseteq \vn(G)$;

(ii) \ $\mathcal T\circ(\Gamma\otimes\Gamma_{\rm op})(\chi) \in L^{\infty}(G) \bar\otimes 1 \bar\otimes L^{\infty}(G)$;

(iii) $\mathcal T\circ(\Gamma\otimes\Gamma_{\rm op}(\chi)) = (\id\otimes\sigma)\big(\chi\otimes 1\big)$;

(iv) $\chi\in V^\infty_{\rm inv}(G)$.
\end{theorem}

\begin{proof}
(i)$\Leftrightarrow$(ii) follows from Corollary \ref{c_com} and the weak* continuity of $\Phi_{\chi}$.

(ii)$\Rightarrow$(iv)
By Lemmas \ref{l:T-in-commutative} and \ref{l_gamma},
\begin{equation}\label{eq_trst}
\mathcal T\circ(\Gamma\otimes\Gamma_{\rm op}(\chi))(s,r,t) =
(\Gamma\otimes\Gamma_{\rm op})(\chi)(s,r,r,t) = \chi_r(s, t),
\end{equation}
for almost all $(s,r,t)\in G\times G\times G$.
Thus, there exist a null set $M\subseteq G$ and a function $\nph\in L^{\infty}(G\times G)$ such that
$$\chi_r = \nph \mbox{ a.e.,} \ \ \mbox{ if } r\not\in M.$$

Let $\xi, \xi_0, \eta, \eta_0\in L^2(G)$ and write $\xi\eta^*$ for the rank one operator on $L^2(G)$ given by
$(\xi\eta^*)(\zeta) = \langle\zeta,\eta\rangle\xi$.
By Lemma \ref{l_trans}, $\nph\in V^{\infty}(G)$ and, whenever $r\not\in M$, we have
\begin{eqnarray*}
\langle \Phi_{\nph}(\xi\eta^*)\xi_0,\eta_0\rangle
& = &
\langle \rho_r\Phi_{\chi}(\rho_r^* (\xi\eta^*) \rho_r) \rho_r^* \xi_0,\eta_0\rangle\\
& = &
\langle \Phi_{\chi}((\rho_r^*\xi)(\rho_r^*\eta)^*) \rho_r^* \xi_0, \rho_r^*\eta_0\rangle.
\end{eqnarray*}
Since the right regular representation is strongly continuous,
the map from $G$ into $\cl B(L^2(G))$, sending $r$ to $(\rho_r^*\xi)(\rho_r^*\eta)^*$, is norm continuous.
Since the family $\{(\rho_r^*\xi)( \rho_r^*\eta)^* : r\in G\}$ is uniformly bounded,
we have that the map
$$r\to \langle \Phi_{\chi}((\rho_r^*\xi)(\rho_r^*\eta)^*) \rho_r^* \xi_0, \rho_r^*\eta_0\rangle$$
is continuous.
Now Lemma \ref{l_trans} implies that
$$\langle \Phi_{\nph}(\xi\eta^*)\xi_0,\eta_0\rangle = \langle \Phi_{\chi_r}(\xi\eta^*)\xi_0,\eta_0\rangle, \ \ \ r\in G,$$
which shows that
$\chi_r = \nph$ almost everywhere for every $r\in G$; in particular, $\nph = \chi$ almost everywhere, 
and (iv) is established.

(iv)$\Rightarrow$(iii)
By (\ref{eq_sth}),
\begin{equation}\label{eq_chirst}
(\id\otimes\sigma)(\chi\otimes 1)(s,r,t) = \chi(s,t), \ \ \mbox{ for almost all } (s,r,t)\in G\times G\times G,
\end{equation}
and the claim now follows from (\ref{eq_trst}).

(iii)$\Rightarrow$(ii) follows from (\ref{eq_chirst}).
\end{proof}

\bigskip

\noindent {\bf Acknowledgement. } We are grateful to Jason Crann for a number of fruitful conversations
on the topic of this  paper.

\end{document}